\documentclass[12pt,leqno]{article}
\usepackage[margin=2.5cm]{geometry}
\usepackage{amsmath,amssymb,amsthm} 

\newtheorem{thm}{Theorem}[section]
\newtheorem{lem}[thm]{Lemma}          
\newtheorem{prop}[thm]{Proposition}   
\newtheorem{cor}[thm]{Corollary}      

\theoremstyle{definition}

\theoremstyle{remark}

\def\ds{\displaystyle}
\def\ol{\overline}
\def\R{\mathbb{R}}
\def\Ri{\mathbb{R}\cup\{+\infty\}}
\def\N{\mathbb{N}}
\def\eps{\varepsilon}

\def\epi{\mathrm{epi\,}}

\def\dom{\mathrm{dom\,}}
\def\dim{\mathrm{dim\,}}
\def\L{\mathbb{L}}
\DeclareMathOperator*{\argmin}{-argmin\,}

\title{Perturbation method for non-convex variational problems}
\author{H. Belcheva\footnote{Sofia University St Kliment Ohridski, Faculty of Mathematics and Informatics, 5 James Bourchier Blvd., 1164 Sofia, Bulgaria, e-mail: htopalova@fmi.uni-sofia.bg}, N. Zlateva\footnote{Sofia University St Kliment Ohridski, Faculty of Mathematics and Informatics, 5 James Bourchier Blvd., 1164 Sofia, Bulgaria, e-mail: zlateva@fmi.uni-sofia.bg}}
\date{May, 2026}

\begin{document}

\maketitle

\begin{abstract}
For a variational problem in a Banach space with quadratic kinetic term and a lower semicontinuous potential, existence of a minimizer is restored by replacing the norm of the space with an equivalent one, arbitrarily close to the original. The form of the problem is preserved; only the norm is perturbed. This extends a result of Ivanov and Zlateva from the convex to the lower semicontinuous case.
\end{abstract}

\textbf{Key words:} variational problem, existence of solutions, per-tur\-bation method, integral functional

\textbf{AMS Subject Classification:} 46N10, 35A15, 49J45, 90C48

\section{Introduction}

Let $(X,\|\cdot\|)$ be a Banach space. We consider the variational problem
\begin{equation}\label{problem-1}
    \left(P_{\|\cdot\|}\right)\left\{\begin{array}{l}
    \ds \inf \int_0^1 \left(\|v(t)\|^2+f(u(t))\right) dt, \\[12pt]
    \ds u(t)=\int_0^t v(s)\, ds, \quad v \in \L_2([0, 1], X),
\end{array}\right.
\end{equation}
where $f:X\to \R\cup\{+\infty\}$ is a proper, lower semicontinuous and bounded below function. This problem can be read as minimizing the action of curve $u$ starting at the origin, with quadratic kinetic term and potential $f$. When $X$ is finite dimensional, the operator $v(t)\mapsto u(t)=\int_0^t v(s)\,ds$ is compact, so weakly convergent minimizing sequences become strongly convergent on the trajectory side, and the direct method of calculus of variations~\cite{clarke_functional_analysis,dacorogna_direct_methods,giusti_direct_methods} delivers a solution. When $X$ is infinite dimensional this fails: the operator is no longer compact, weak convergence of $v_n$ does not imply norm convergence of $u_n$, and even simple choices of $f$ may leave $(P_{\|\cdot\|})$ without a minimizer.

The standard remedy is a variational principle, for example that of Ekeland, to perturb the functional by an arbitrarily small term and recover existence. The price is that the perturbed functional no longer has the original form of $(P_{\|\cdot\|})$. An alternative approach, developed by Ivanov and Zlateva~\cite{iv_var_integr,iz_integr_funct}, is to perturb only the integrand of the integral functional, thus preserving the structure of the problem. For convex $f$, this leads to the following result, easily derivable from \cite[Theorem~1.1]{iv_var_integr}.

\begin{thm}
\label{thm:iz}
    Let $(X,\|\cdot\|)$ be a Banach space. Let $f: X \to \mathbb{R} \cup\{+\infty\}$ be a convex lower semicontinuous function such that $0\in \dom f$, and $f\geq k\|\cdot\|$ for some constant $k>0$. Consider the optimization problem $\left(P_{\|\cdot\|}\right)$ defined by \eqref{problem-1}.

    For each $\varepsilon>0$ there is an equivalent norm $|\cdot|$ on $X$ such that
$$
\|\cdot\| \le |\cdot| \le (1+\varepsilon)\|\cdot\|
$$
and the corresponding problem $\left(P_{|\cdot|}\right)$ has a solution.
\end{thm}

The conclusion is striking: existence is restored not by adding a perturbation term to the functional, but by replacing the norm on $X$ with an equivalent one, arbitrarily close to the original. The problem $(P_{|\cdot|})$ has \emph{exactly the same form} as $(P_{\|\cdot\|})$ -- a quadratic kinetic term plus the same potential $f$ along the trajectory -- it is merely posed in $(X,|\cdot|)$ instead of $(X,\|\cdot\|)$.

The convexity assumption on $f$, however, is restrictive. The aim of the present work is to remove it. Combining techniques from~\cite{iv_var_integr,iz_integr_funct,Orlitz} with a tailored variational principle in the spirit of Deville--Godefroy--Zizler, we extend the above theorem to the case of a merely lower semicontinuous $f$, see Theorem~\ref{main}.

\medskip

The paper is organized as follows. Section~\ref{sec:prelim} fixes notations, recalls the Bochner integral setting, and discusses why the direct method fails in infinite dimensions. In Section~\ref{sec:topol} we introduce an intermediate topology $\tau$ on $Y=\L_2([0,1],X)$, which is finer than the weak topology but coarser than the norm topology, and prove a generalized Cantor Lemma adapted to it. Section~\ref{sec:pert} sets up the perturbation space of squared continuous seminorms on $X$. In Section~\ref{sec:VP} we adapt the variational principle of Deville--Godefroy--Zizler~\cite{DGZ,DGZ-book} to our setting, see Theorem~\ref{VP}. Section~\ref{sec:main} contains the main result, Theorem~\ref{main}. Section~\ref{sec:concl} outlines directions for further research.

 \section{Preliminaries and notations}\label{sec:prelim}

Let us consider first the finite-dimensional case of the problem  $\left(P_{\|\cdot\|}\right)$, that is when $X$ is a finite dimensional Euclidean space $E$,  and $\|\cdot\|$ is the Euclidean norm.
For $v\in  \L_2([0,1], E)$ denote
$$
  \varphi (v)(t):=\int_0^t v(s) \, ds.
$$
Considered as an operator from $\L_2([0,1], E)$ into itself, $\varphi$ is  a compact linear operator, so it maps weakly convergent sequences into strongly convergent ones. Let $\{v_n\}_{n=1}^{\infty}$ be a minimizing sequence to the problem
$$
\left(P_{\|\cdot\|}\right)\quad  \inf \int_0^1 (\|v(t)\|^2+f(\varphi(v)(t))) \, dt .
$$
Since $\{v_n\}_{n=1}^{\infty}$ will be bounded, and since the unit ball of $\L_2([0,1], E)$ is weakly compact, there will be  a subsequence,  that we will denote again by $\{v_n\}_{n=1}^{\infty}$, such that $v_n$ weakly converges to some $\bar v$, as $n \rightarrow \infty$, written as $v_n\to_w \bar v$. Then $\varphi (v_n) $ norm converges to $\varphi(\bar v)$,  as $n \rightarrow \infty$, written as
$
    \varphi (v_n) \rightarrow \varphi(\bar v)$,
and $\bar v$ is a solution to $(P_{\| \cdot \|})$.

\medskip

The above illustrates the so-called direct method of calculus of variations, see e.g. \cite{clarke_functional_analysis,dacorogna_direct_methods,giusti_direct_methods}. It works, because the value function is weakly lower semicontinuous.

In the case of a general Banach space $X$, however, and even for a Hilbert space $H$,  the operator $\varphi$ is not weak-to-norm continuous.
Indeed, let for example  $H=\ell_2$, and
$$
    v_n(t)=e_n, \quad \forall t \in[0,1],
$$
where $\{e_n\}_{n=1}^\infty$ is the canonical basis of $\ell_2$.
So,  $v_n \to_w 0$, but
$$
    \varphi(v_n)(t)=t e_n, \quad \forall n \in \mathbb{N},\ \forall t\in[0,1],
$$
and, since
$$
    \|e_i-e_j\|=\sqrt{2}, \quad \forall i \neq j,
$$
we have
$$
   \|\varphi(v_i)-\varphi(v_j)\|^2=2/3, \quad \forall i \neq j,
$$
hence the set $\left\{\varphi\left(v_n\right): n \in \mathbb{N}\right\}$ is discrete.

\bigskip

 We work in a Banach space $(X,\|\cdot\|)$ and we will denote $
    Y := \L_2([0,1],X)$.

All integrals are considered in Bochner sense, see \cite[Chapter~II]{diestel_uhl_vector}.
More precisely, see e.g.~\cite[p.~8]{hytonen_analysis_banach}, a function $s:[0,1]\to X$ is \textit{simple} if
$$
    s(t) = \sum_{i=1}^n c_i \chi_{A_i}(t),
$$
for some $c_i\in X$ and some Lebesgue measurable sets $A_i\subset[0,1]$, where $\chi_A$ is the \textit{characteristic function} of $A$:
$$
    \chi_A(t) := \begin{cases} 1, & t\in A,\\ 0, & t\notin A. \end{cases}
$$
A function $v:[0,1]\to X$ is \textit{strongly measurable} if there is a sequence $\{s_n\}_{n=1}^\infty$ of simple functions such that
$$
   \lim_{n\to\infty}\|v(t) - s_n(t)\| = 0,\text{ for almost all }t\in[0,1].
$$
A strongly measurable function $v$ is \emph{Bochner integrable} if there exists a sequence $\{s_n\}_{n=1}^\infty$ of simple functions such that
$$
    \lim_{n\to\infty}\int_0^1 \|v(t) - s_n(t)\|\, dt =0,
$$
see \cite[Definition~1.2.1]{hytonen_analysis_banach}. The integral of $v$ is then the limit of the integrals of the $s_n$'s while for a simple function $s$ the definition of integral is clear:
$$
    \int_0^1s(t)\, dt = \sum_{i=1}^n c_i m(A_i),
$$
where $m$ denotes the Lebesgue measure.

Alternatively, by \cite[Proposition~1.2.2]{hytonen_analysis_banach}, a strongly measurable function $v$ is Bochner integrable if and only if
$$
    \int_0^1\|v(t)\|\, dt < \infty.
$$
 The space $Y$ then consists of those Bochner integrable functions $v$ for which
$$
    \|v\|_{\L_2}^2 := \int_0^1\|v(t)\|^2\, dt < \infty,
$$
see \cite[Definition~1.2.15]{hytonen_analysis_banach}.

\medskip

We will also use some auxiliary results.
\begin{lem}[{\cite[Lemma~1.2.19]{hytonen_analysis_banach}}]
    \label{lem:simple-dense}
    The simple functions are dense in $Y$.
\end{lem}

The  linear operator $\varphi:\L_2([0,1],X)\to \L_2([0,1],X)$
retains the following important property.
\begin{lem}
    \label{lem:fi-w-to-norm-closed}
    The bounded linear operator $
        \varphi: \L_2([0,1], X) \to \L_2([0,1], X)$ defined as
\begin{equation}\label{def:phi}
  \varphi (v)(t):=\int_0^t v(s) \, ds,\quad \forall v\in  \L_2([0,1],X)
\end{equation}
    is sequentially weak-to-norm closed.
\end{lem}

\begin{proof}
     Note that for each $v\in Y$ the function $
        t \to \varphi(v)(t)
    $
    is continuous. Let
    $
          v_n \to_w  v$ (i.e. in the weak topology in $Y$) and $\varphi(v_n) \to \ol u$ (i.e. in the norm topology in $Y$).
    We have to show that $\ol u = \varphi(\bar v)$.

    Fix an arbitrary $x^*\in X^*$ and an arbitrary $t\in[0,1]$. The linear functional
    $$
        v \to \int_0^t \langle x^*, v(s)\rangle \, ds
    $$
    belongs to  $Y^*$, so from the weak convergence of $v_n$ to $\bar v$, we have that
    $$
        \lim_{n\to\infty} \int_0^t \langle x^*, v_n(s)\rangle \, ds= \int_0^t \langle x^*, \bar v(s)\rangle \, ds,
    $$
    or,
    $$
        \lim_{n\to\infty} \left\langle x^*, \int_0^t v_n(s) \, ds\right\rangle = \left\langle x^*, \int_0^t \bar v(s) \, ds\right\rangle.
    $$
    Therefore,
    $$
        \lim_{n\to\infty} \langle x^*,\varphi(v_n)(t)\rangle = \langle x^*,\varphi(\bar v)(t)\rangle, \quad \forall x^*\in X^*,\ \forall t\in[0,1],
    $$
    and,
    \begin{equation}
        \label{eq:fi-weak}
        \varphi(v_n)(t) \to_w \varphi(\bar v)(t), \text{ as }n\to \infty, \quad \forall t\in[0,1].
    \end{equation}
    But the norm convergence of $\varphi(v_n)$ to $\ol u$ in $Y$, that is,
    $$
        \lim_{n\to\infty} \int_0^1 \|\varphi(v_n)(t) - \ol u(t)\|^2\, dt = 0,
    $$
    and the continuity of the functions $\varphi(v_n)$, and $\ol u$ give
    $$
         \varphi(v_n)(t) \to \ol u(t), \text{ as }n\to \infty, \quad \forall t\in[0,1].
    $$
    The latter and \eqref{eq:fi-weak} yield the desired conclusion $\ol u = \varphi(\bar v)$.
    \end{proof}

The function $f:X\to\Ri$ is called lower semicontinuous if its \emph{epigraph} $\epi f:=\{ (x,r)\in X\times \R: f(x)\le r\}$ is a closed set, and \emph{proper} if its \emph{domain} $\dom f:=\{ x\in X:f(x)<\infty\}$ is non-empty.

Let $f: X \to \mathbb{R} \cup \{+\infty\}$ be a fixed proper, lower semicontinuous and bounded below  function. We will assume without loss of generality that $f\ge0$.

Let $P$ denotes the closed cone of all continuous semi-norms on $X$ squared. That is, each $p\in P$ is positive, convex, continuous and $2$-homogeneous function:
$$
    p(tx) = t^2p(x),\quad\forall x\in X,\ \forall t\in\mathbb{R}.
$$
With the metric induced by the uniform convergence on the unit ball $B_X$ of $X$,
$$
    \bar\rho(p_1,p_2) := \| {p_1}- {p_2}\|_\infty := \sup_{x\in B_X} | {p_1(x)} -  {p_2(x)}|,
$$
the cone $P$ turns into a complete metric space $(P,\bar\rho)$.

Consider for $p\in P$ the function $H_p:\L_2([0,1],X) \to \mathbb{R} \cup \{+\infty\}$ defined as
$$
    H_p(v) := K_p(v) + F(v),
$$
where
$$
    K_p(v) := \int_0^1 \|v(t)\|^2\, dt + \int_0^1 p(v(t))\, dt, \text{ and }
    F(v) := \int_0^1 f(\varphi(v)(t))\, dt.
$$
Obviously, we mean the (integrals of) squared seminorms $p$  as perturbations of the main problem $(P_{\|\cdot\|})$ which might be rewritten in the form
\begin{equation}\label{problem}
(P_{\|\cdot\|})\qquad \inf_{v\in Y}\  \int_0^1\left(\|v(t)\|^2+f(\varphi(v)(t))\right) \, d t.
\end{equation}
For simplicity we also write $K := K_0$, and $
    H := H_0$.
In our notation the problem $(P_{\|\cdot\|})$ is equivalent to:
$$
 \inf_{v\in Y}   H(v) .
$$
We aim to show that for each $\varepsilon > 0$ there is $p\in P$ such that $\bar\rho(p,0) = \|p\|_\infty < \varepsilon$, and $H_p$ attains minimum.

Another auxiliary result we shall use in the sequel is the well-known Fatou Lemma.

\begin{lem}[{\cite[Theorem~1.28]{rudin_real_complex}}]\label{lem:fatu}
    Let $g_n: [0,1] \to [0,+\infty]$, $n\in \N$ be strongly measurable functions. Then
    $$
        \int_0^1 \liminf_{n\to\infty} g_n(t) \, dt \le \liminf_{n\to\infty} \int_0^1 g_n(t) \, dt.
    $$
\end{lem}
From the Fatou Lemma it easily follows that the functions $F$, $K_p$,  and  therefore, $H_p$ are norm lower semicontinuous for each $p\in P$. To see this for e.g. $F$, fix $v\in Y$ and a sequence $\{v_n\}_{n=1}^\infty\subset Y$ such that
$$
  v_n\to v\text{ and } \lim_{n\to\infty} F(v_n)= \liminf_{w\to v} F(w).
$$
Then $v_n \to_w v$ and from Lemma~\ref{lem:fi-w-to-norm-closed} it holds that
$
  \varphi(v_n)\to \varphi(v)$.
The latter spells
$$
    \lim_{n\to\infty} \int_0^1 \|\varphi(v) (t) -\varphi(v_n) (t)\|^2\, dt = 0,
$$
which implies that for a subsequence, still denoted $\{v_n\}_{n=1}^\infty $, see \cite[Theorem~3.12]{rudin_real_complex},
$$
    \varphi(v_n) (t) \to  \varphi(v) (t),\text{ for almost all }t\in[0,1].
$$
Set
$
    g(t) := f(\varphi(v)(t))$, $g_n(t) := f (\varphi(v_n)(t))$.
Since $f$ is norm lower semicontinuous,
$$
    \liminf_{n\to\infty} g_n(t) \ge  g(t),\text{ for almost all }t\in[0,1].
$$
By the Fatou Lemma,
\begin{align*}
    \lim_{n\to\infty} F(v_n) &= \lim_{n\to\infty} \int_0^1 g_n(t) \, dt\\
    &\ge \int_0^1 \liminf_{n\to\infty} g_n(t) \, dt \ge \int_0^1 g(t) \, dt\\[10pt]
    &= F(v),
\end{align*}
which means that $F$ is norm lower semicontinuous at $v$.

In fact the function $K_p$ is even weakly lower semicontinuous, because it is convex, but $F$ is not and that is why we need to introduce an intermediate topology in $Y$.

\section{The topology $\tau$}\label{sec:topol}

    Apart from the norm and weak topologies on $Y := \L_2([0,1],X)$ we consider also an intermediate linear topology $\tau$ on $Y$, which is the product of the weak and norm topologies on the graph of the linear operator $\varphi$ defined by \eqref{def:phi}. The graph of $\varphi$  is a closed set in $(Y,\sigma(Y^*,Y))\times (Y,\|\cdot\|_{\L_2})$, hence we can identify the latter with $Y$ via the mapping $v\mapsto (v,\varphi(v))$. Since we are working with sequences, what we need in most cases is:
    $$
         v_n\to_\tau v \iff
             v_n \to_w v, \text{ and }
             \varphi(v_n) \to \varphi(v).
          $$
    Clearly, $\tau$ is stronger than the weak topology on $Y$ and weaker than the norm topology.

    \begin{prop}
        For each $p\in P$ the function $H_p$ is sequentially $\tau$ lower semicontinuous.
    \end{prop}
    \begin{proof}
       Let us fix $p\in P$.
        We have already noted that $K_p$ is weakly lower semicontinuous, and hence $\tau$ lower semicontinuous. To check that $F$   is  sequentially $\tau$  lower semicontinuous, let $v_n\to_\tau v$, as $n\to\infty$. In particular, $\varphi(v_n)\to  \varphi(v)$ as $n\to\infty$, and as in Section~\ref{sec:prelim}, we see that
        $\ds
            \liminf_{n\to\infty} F(v_n) \ge F(v)$.
    \end{proof}

    \begin{prop}\label{seq-tau-lsc}
        A subset $A\subset Y$ is sequentially $\tau$-compact if and only if $A$ is weakly-compact and $\varphi(A)$ is norm-compact.

        In particular, if $X_1$ is a finite-dimensional subspace of $X$ and
        $
            Y_1 := \L_2([0,1],X_1)$,
               then $B_{Y_1}$ is sequentially $\tau$-compact.
    \end{prop}
    \begin{proof}
        Let $A\subset Y$ be a sequentially $\tau$-compact set, and let the sequence $\{v_n\}_{n=1}^\infty\subset A$ be arbitrary. Since $A$ is sequentially  $\tau$-compact, there is a subsequence $\{v_{n_k}\}_{k=1}^\infty$ such that
        $$
             v_{n_k} \to_\tau v \in A\text{ as }k\to \infty.
        $$
        In particular, $v_{n_k} \to_w v$, as $k\to\infty$ and $A$ is sequentially weakly compact, which is the same as weakly compact by the Eberlein-\v{S}mulian theorem, cf. e.g. \cite[Theorem~3.28]{rudin_functional}.

        To see that $\varphi(A)$ is norm-compact, let $\{u_n\}_{n=1}^\infty \subset \varphi(A)$ be arbitrary. For each $n\in\N$ choose $w_n\in A$ such that $\varphi(w_n)=u_n$. Since $A$ is sequentially $\tau$-compact, there is a subsequence $\{w_{n_k}\}_{k=1}^\infty$ and $w\in A$ such that $w_{n_k}\to_\tau w$, as $k\to\infty$. In particular, $\varphi(w_{n_k}) \to \varphi(w)$ in the norm of $Y$, that is, $u_{n_k}\to \varphi(w)\in \varphi(A)$, as $k\to\infty$. Hence $\varphi(A)$ is sequentially norm-compact, and therefore norm-compact.

        On the other hand, let $A\subset Y$ be such that $A$ is weakly compact and $\varphi(A)$ is norm-compact. Let the sequence $\{v_n\}_{n=1}^\infty\subset A$ be arbitrary. Since $A$ is weakly compact, and thus sequentially weakly  compact by the Eberlein-\v{S}mulian theorem, there is a subsequence $\{v_{n_k}\}_{k=1}^\infty$ such that $v_{n_k} \to_w v \in A$, as $k\to\infty$. Since $\varphi(A)$ is norm-compact, there is a further subsequence $\{v_{n_{k_i}}\}_{i=1}^\infty$ such that $\varphi(v_{n_{k_i}}) \to \bar v \in \varphi(A)$, as $i\to\infty$. Since the graph of $\varphi$ is sequentially closed in the product of the weak and norm topologies, see Lemma~\ref{lem:fi-w-to-norm-closed}, we have $\bar v = \varphi(v)$ and, therefore, $v_{n_{k_i}} \to_\tau v$, as $i\to\infty$. Hence, $A$ is sequentially $\tau$-compact.

        Let now $X_1\subset X$ be a finite-dimensional subspace, and  $Y_1 := \L_2([0,1],X_1)$. Since $Y_1$ is reflexive, $B_{Y_1}$ is weakly-compact, i.e. $\sigma(Y_1^*,Y_1)$-compact. The $Y^*$-weak topology on $Y_1$ when considered as a subspace of $Y$, that is the topology $\sigma(Y^*,Y_1)$, is weaker than its own $Y_1^*$-weak topology $\sigma(Y_1^*,Y_1)$, by the Hahn-Banach Theorem \cite[Theorem~3.3]{rudin_functional}. Therefore, $B_{Y_1}$ is $\sigma(Y^*,Y)$-compact, that is, weakly compact in $Y$. Also, since $\dim X_1 < \infty$, by the Arzel\`a-Ascoli Theorem \cite[Theorem~A.5]{rudin_functional} the restriction of $\varphi$ on its invariant subspace $Y_1$ is compact, so $\varphi (B_{Y_1})$ is norm-compact. From the already established characterization, it follows that $B_{Y_1}$ is sequentially $\tau$-compact.
    \end{proof}

 Kuratowski~\cite{Kuratowski} proved a generalization of the Cantor Lemma in terms of the $\alpha$ measure of non-compactness he defined. This line was extended by De Blasi~\cite{De_Blasi} who defined measure $\beta$ of weak non-compactness of a set in a Banach space.   The following result is another generalization of  the Cantor Lemma.

\begin{lem}[\textbf{Generalized Cantor Lemma}]
    \label{lem:gen-cantor}
    Let $(X,\bar\tau)$ be a topological space and let $d$ be a sequentially $\bar\tau$ lower semicontinuous metric on $X$ such that $(X,d)$ is a complete metric space and, moreover,
    \begin{equation}
        \label{eq:tops-rshp}
        \forall x\in U \in \bar\tau\ \exists \varepsilon > 0,V\in\bar\tau:\ x\in V\text{ and } V_\varepsilon\subset U,
    \end{equation}
    where $V_\varepsilon := \{x\in X:\ d(x,y)\le\varepsilon$ for some $y\in V\}$.

    Let for a set $A\subset X$, $\bar\mu (A)$ denotes the measure of $\bar \tau$ non-compactness of $A$, that is
    $$
        \bar\mu (A) := \inf\{\varepsilon:\ A \subset C_\varepsilon,\text{ for some sequentially }\bar\tau \text{-compact set } C\}.
    $$

    If the sets $A_n\subset X$ are non-empty, nested, i.e. $A_{n+1}\subset A_n$ for all $n\in\mathbb{N}$, sequentially $\bar\tau$-closed, and such that
    $\ds
        \lim_{n\to\infty} \bar\mu(A_n) = 0$,
        then
    $$
        \bigcap_{n=1}^\infty A_n \neq \varnothing.
    $$
\end{lem}
\begin{proof}
    For any $n\in \N$, fix
    $
        x_n \in A_n $.

    Because the sets $A_n$ are nested and sequentially $\bar\tau$  closed, it is enough to show that the sequence $\{x_n\}_{n=1}^\infty$ has a $\bar\tau$-convergent subsequence. To find the latter, we will perform a Cantor diagonalization procedure.

    Let $\varepsilon_n > 0$ be such that
    $$
        \varepsilon_{n+1} < \varepsilon_n,\ \forall n\in\mathbb{N},\text{ and }\lim_{n\to\infty} \varepsilon_n = 0,
    $$
    and for some sequentially $\bar\tau$-compact sets $C_n$, $
        A_n \subset (C_n)_{\varepsilon_n}$.

    Let
    $
        y_n^{(1)} \in C_1$ be such that $d(x_n,y_n^{(1)}) \le \varepsilon_1$.

    Since $C_1$ is sequentially $\bar\tau$-compact, there is $\bar\tau$-convergent subsequence of $\{y_n^{(1)}\}_{n=1}^\infty$. That is, there is strictly increasing function $\sigma_1:\mathbb{N}\to\mathbb{N}$ such that
    $$
         y_{\sigma_1(n)}^{(1)} \to_{\bar\tau} z_1\text{ as } n\to \infty.
    $$
    We choose subsequence of subsequence by induction. Let $\{y_{\sigma_k(n)}^{(k)}\}_{n=1}^\infty\subset C_k$ be such that
    $$
         y_{\sigma_k(n)}^{(k)}\to_{\bar\tau} z_k \text{ as } n\to \infty.
    $$
    Since $\{x_{\sigma_k(n)}\}_{n=1}^\infty$ is a subsequence of $\{x_n\}_{n=1}^\infty$, we have that
    $$
        x_{\sigma_k(n)} \in A_{k+1} \subset (C_{k+1})_{\varepsilon_{k+1}},\quad\forall n \ge k + 1.
    $$
    So, there are
    $$
        \{y_{\sigma_k(n)}^{(k+1)}\}_{n=k+1}^\infty \in C_{k+1}\text{ such that }d(x_{\sigma_k(n)},y_{\sigma_k(n)}^{(k+1)}) \le \varepsilon_{k+1}.
    $$
    Since $C_{k+1}$ is sequentially $\bar\tau$-compact, there is a $\bar\tau$-convergent subsequence of $\{y_{\sigma_k(n)}^{(k+1)}\}_{n=k+1}^\infty$, that is, there is a strictly increasing function $\nu_k$ from $k+1,k+2,\ldots$ into itself such that for
    $
        \sigma_{k+1}(n) := \sigma_k(\nu_k(n))$
    we have that
    $$
         y_{\sigma_{k+1}(n)}^{(k+1)} \to_{\bar\tau} z_{k+1}\text{ as } n\to \infty,
    $$
    and so on \ldots.

    We claim that the sequence $\{z_k\}_{k=1}^\infty$ is  convergent in the metric $d$. Indeed, for any $j>k$, because of the inclusion $\{\sigma_j(n):\ n\in\N\} \subset \{\sigma_k(n):\ n\in\N\}$, we have that
    $$
        d(x_{\sigma_j(n)},y_{\sigma_j(n)}^{(k)}) \le \varepsilon_k.
    $$
    But also, by construction,
    $$
        d(x_{\sigma_j(n)},y_{\sigma_j(n)}^{(j)}) \le \varepsilon_j < \varepsilon_k.
    $$
    So, by the triangle inequality, $d(y_{\sigma_j(n)}^{(k)},y_{\sigma_j(n)}^{(j)}) \le 2 \varepsilon_k$.

    Using that
    $
         y_{\sigma_j(n)}^{(k)} \to_{\bar\tau} z_k$, and $  y_{\sigma_j(n)}^{(j)} \to_{\bar\tau} z_j$, as $n\to \infty$,
        and the $\bar\tau$ lower semicontinuity of the metric $d$, we get
    $$
        d(z_j,z_k) \le 2 \varepsilon_k,\quad\forall j \ge k,
    $$
    which means that the sequence $\{z_k\}_{k=1}^\infty$ is a $d$-Cauchy sequence, and because $(X,d)$ is complete, we get a $z$ such that
    $$
         d(z_k,z) \to 0\text{ as } k\to \infty.
    $$

    Let $U\in\bar\tau$ be an arbitrary neighbourhood of $z$, that is, $z\in U$. Let $V\in\bar\tau$ and $\varepsilon>0$ be such that $z\in V$ and $V_\varepsilon\subset U$, see \eqref{eq:tops-rshp}. Let $j\in\mathbb{N}$ be so large that
    $$
        \varepsilon_j < \varepsilon\text{ and }z_j\in V.
    $$
    Now, since $
        \{\sigma_k(k):k\in\mathbb{N},\ k\ge j\} \subset \{\sigma_j(n): n\in\mathbb{N}\}$,
    and since $y_{\sigma_j(n)}^{(j)} \to_{\bar\tau} z_j$, as $n\to \infty$, there is $N\in\mathbb{N}$ such that
    $$
        y^{(j)}_{\sigma_k(k)} \in V,\quad \forall k > N.
    $$
    But $d(x_{\sigma_k(k)},y^{(j)}_{\sigma_k(k)}) \le \varepsilon_j < \varepsilon$, therefore
    $$
        x_{\sigma_k(k)} \in V_\varepsilon \subset U,\quad \forall k > N.
    $$
    This means that $\{x_{\sigma_k(k)}\}_{k=1}^\infty$ $\bar\tau$-converges to $z$, and the proof is  completed.
\end{proof}

We will use the following particular case of the generalized Cantor Lemma for the  topology $\tau$.
\begin{cor}\label{cor:Cantor}
Let $S$ be a nonempty closed set in $Y$.
    If the sets $A_n\subset S$ are non-empty, nested, i.e. $A_{n+1}\subset A_n$ for all $n\in\mathbb{N}$, sequentially $\tau$-closed, and such that
    $\ds
        \lim_{n\to\infty} \mu(A_n) = 0$,
where $\mu (A)$ is the measure of $ \tau$ non-compactness of $A$, that is
   \begin{equation}\label{def:mu}
        \mu (A) := \inf\{\varepsilon:\ A \subset C_\varepsilon,\text{ for some sequentially }\tau \text{-compact set } C\},
 \end{equation}
    then
    $$
        \bigcap_{n=1}^\infty A_n \neq \varnothing.
    $$
\end{cor}
\begin{proof}
We consider the topological space $(S,\tau)$ and take $d$ to be the metric on $S$ induced by the norm in $Y$. Thus $(S, \|\cdot\|_{\L_2})$ is a complete metric space. Since the norm $\|\cdot\|_{\L_2}$ is a convex and continuous function, it is weakly lower semicontinuous and, therefore, sequentially $\tau$ lower semicontinuous, and so is the metric $d$ on $S$.

Condition \eqref{eq:tops-rshp} holds because $\tau$ is a linear topology on $Y$. Indeed, let $x\in U\in \tau$. By continuity of addition at $(x,0)$, there exists a $\tau$-open neighbourhood $V$ of $x$ and a $\tau$-open neighbourhood $W$ of $0$ such that $V+W\subset U$. Since $\tau$ is weaker than the norm topology, $W$ is norm-open, and being a neighbourhood of $0$, contains some $\varepsilon B_Y$ with $\varepsilon>0$. Then $V_\varepsilon = V+\varepsilon B_Y\subset V+W\subset U$. Obviously, taking intersections with $S$ changes nothing.

From the generalized Cantor Lemma the claim follows.
\end{proof}

\section{Perturbation space}\label{sec:pert}

    To construct a perturbation space on $Y$, we ``lift'' each function $P\ni p:X\to \R$ to a function  $\mathfrak{p}: Y \to \mathbb{R}$ by
    $$
        \mathfrak{p}(v) := \int_0^1 p(v(t))\, dt.
    $$
    We consider the lifted cone
    $
        \mathcal{P} := \{ \mathfrak{p}:\ p\in P\}
    $
    with the lifted metric
    $$
         \rho (\mathfrak{p}_1,\mathfrak{p}_2) := \bar\rho (p_1,p_2),
    $$
    so $(\mathcal{P},\rho)$ is a complete metric space.

    We will show that $(\mathcal{P},\rho)$ is a perturbation space on any ball $rB_Y$, where $B_Y$ is the unit ball of $Y$, with respect to the measure $\mu$ of $\tau$-non-compactness defined by \eqref{def:mu}.

For a definition of a perturbation space on a set, see \cite[Definition 2]{Orlitz}.

    To this end we will prove first some auxiliary results.

    \begin{lem}
        \label{lem:42}
        Let $\varepsilon > 0$ and $v\in Y$ be arbitrary. There exists a finite dimensional $X_1\subset X$ such that
        $$
            d(v,Y_1) < \varepsilon,\text{ where }Y_1 = \L_2([0,1],X_1).
        $$
    \end{lem}
    \begin{proof}
        Fix $\varepsilon > 0$ and $v\in Y$. From Lemma~\ref{lem:simple-dense} there is a simple function
        $$
            s(t) = \sum_{i=1}^n c_i \chi_{A_i}(t),
        $$
        such that $\|v-s\|_Y < \varepsilon$. Let
        $
            X_1 := \mathrm{span}\,\{c_i:\ i=1,\ldots,n\}$,
            and $Y_1 := \L_2([0,1],X_1)$. Clearly $s\in Y_1$, so $d(v,Y_1) \le \|v-s\|_Y< \varepsilon$.
    \end{proof}
    \begin{lem}[{\cite[Lemma~2.5]{iv_var_integr}}]
        \label{lem:43}
        Let $X_1$ be a finite-dimensional subspace of $X$ and $Y_1 = \L_2([0,1],X_1)$. Then for any $v\in Y$,
        $$
            d^2(v,Y_1) = \int_0^1 d^2(v(t),X_1)\, dt.
        $$
    \end{lem}

 \begin{prop} \label{prop:dominate-uniform}
 $({\cal{P}},\rho)$ is a complete metric space. For any set $S=rB_Y$, $r>0$, the metric $\rho$ dominates the uniform convergence on $S$, i.e.
 \[
 \sup_{v\in S}\mathfrak{p}(v)\le r^2\rho(\mathfrak{p},0).
 \]
 \end{prop}

 \begin{proof}
  $({\cal{P}},\rho)$ is a complete metric space because  $(P,\bar\rho)$ is a complete metric space.

  Fix $r>0$ and let $S=rB_Y$. Let $\mathfrak{p}\in {\cal{P}}$ be arbitrary, and let $\ds \mathfrak{p}(v) := \int_0^1 p(v(t))\, dt$ for some $p\in P$. We have that
 \begin{eqnarray*}
  \mathfrak{p}(v)&=&  \int_0^1 p(v(t))\, dt =   \int_0^1 \|v(t)\|^2 p\left(\frac{v(t)}{\|v(t)\|}\right)\, dt\\
  &\le& \int_0^1 \|v(t)\|^2 \bar\rho(p,0)\, dt = \bar\rho(p,0)\int_0^1 \|v(t)\|^2 \, dt =\rho(\mathfrak{p},0)\|v\|^2.
 \end{eqnarray*}
Therefore,
\[
 \sup_{v\in rB_Y}\mathfrak{p}(v)\le \rho(\mathfrak{p},0)\sup_{v\in rB_Y} \|v\|^2 =r^2\rho(\mathfrak{p},0).
 \]
 and the proof is completed.
 \end{proof}

    \begin{prop}
        \label{prop:41}
        For every $\varepsilon > 0$ there exists $\delta>0$ such that for each $v_0\in Y$   there exists $\mathfrak{p}\in \mathcal{P}$ (depending on $v_0$)   such that
      \[  \begin{aligned}
            (a)\quad & \rho(\mathfrak{p}, 0) < \varepsilon, \\
            (b)\quad & v_0 \in \delta\argmin  \mathfrak{p}, \\
            (c)\quad & \mu((3\delta)\argmin \mathfrak{p}) < \varepsilon.
        \end{aligned}\]
    \end{prop}
    \begin{proof}
        Obviously, $\min\mathfrak{p} = \mathfrak{p}(0) = 0$ for any $\mathfrak{p} \in \mathcal{P}$, so
        $$
            \delta\argmin \mathfrak{p} = \{v\in Y:\ \mathfrak{p}(v) \le \delta\}.
        $$

        Fix an $\varepsilon > 0$ and then fix a $\delta > 0$ such that
        $            \delta < \varepsilon^3 / 12$. Fix $v_0\in Y$.
                Then fix an $a > 0$ such that
        $
            \varepsilon/2 > a > 6 \delta / \varepsilon^2$, and fix  some $b \in (0,\varepsilon/2)$ such that
        $
            b\|v_0\|^2 < \delta / 2$.

        From Lemma~\ref{lem:42} we can find a finite-dimensional subspace $X_1 \subset X$ such that for $Y_1 = \L_2([0,1],X_1)$,
        $$
            d(v_0,Y_1) < \sqrt\frac{\delta}{2a}.
        $$
        Denote
        $
            C := 2\sqrt{3\delta/b} B_{Y_1}$.
        From Proposition~\ref{seq-tau-lsc}, the set $C$ is sequentially $\tau$-compact.

        Define $p\in P$ as
        $$
            p(x) := a d^2(x,X_1) + b \|x\|^2.
        $$
        From Lemma~\ref{lem:43} we have for the corresponding to it function $\mathfrak{p} \in \mathcal{P}$ that
        $$
            \mathfrak{p}(v) = a d^2(v,Y_1) + b \|v\|^2.
        $$
        Hence, $\rho(\mathfrak{p},0) = \bar\rho(p,0) = \sup p(B_X) \le a + b  < \varepsilon$ and $(a)$ is verified.

       Since
        $$
            \mathfrak{p}(v_0) = a d^2(v_0,Y_1) + b \|v_0\|^2 < a \frac{\delta}{2a} +  \frac{\delta}{2} = \delta,
        $$
       we have that $v_0 \in \delta\argmin \mathfrak{p}$ and $(b)$ holds.

        Let now $v \in (3\delta)\argmin \mathfrak{p}$ be arbitrary. That is, $\mathfrak{p} (v) \le 3\delta$. The latter implies that
        $                a d^2(v,Y_1) \le 3\delta$, as well as $  b \|v\|^2 \le 3\delta$.
              Hence, $v\in \sqrt{3\delta/b}B_Y$. We claim that
        $$
            d(v,Y_1) = d(v,C).
        $$
    Indeed, since $C\subset Y_1$, we have $d(v,Y_1) \le d(v,C)$. For the opposite inequality, let $w_0\in Y_1$ be the nearest point to $v$ in $Y_1$, i.e. $d(v,Y_1) = \|v-w_0\|$. If $w_0\notin C$, then $\|w_0\| > 2\sqrt{3\delta/b}$, so
    $$
        \|v-w_0\| \ge \|w_0\| - \|v\| > 2\sqrt{3\delta/b} - \sqrt{3\delta/b} = \sqrt{3\delta/b},
    $$
    contradicting $v\in \sqrt{3\delta/b}\, B_Y$. Hence $w_0\in C$, and $d(v,C)\le \|v-w_0\| = d(v,Y_1)$.

 Using that
       $d^2(v,Y_1) \le 3\delta/a < \varepsilon^2$,   thus $ d(v,Y_1) \le  \varepsilon $ and  $d(v,Y_1)=d(v,C)$  we obtain that $v\in C_\varepsilon$.
        Since $v \in (3\delta)\argmin \mathfrak{p}$ was arbitrary,  $
            (3\delta)\argmin  \mathfrak{p} \subset C_\varepsilon $
        and $(c)$ holds.
    \end{proof}

    Propositions~\ref{prop:dominate-uniform} and \ref{prop:41} show that the complete metric space $(\mathcal{P},\rho)$ exhibits  all features of a perturbation space on each ball in $Y$ with respect to the measure $\mu$ of $\tau$ non-compactness, see \cite[Definition 2]{Orlitz}. Perturbation spaces with respect to the Kuratowski measure of non-compactness $\alpha$ (see \cite[Definition 4.2]{iz_integr_funct}, \cite[Definition 2]{Orlitz}) or with respect to De Blasi measure of weak non-compactness $\beta$ (see \cite[Definition 3.1]{iv_var_integr}) are used for establishing variational principles in the corresponding papers \cite{iz_integr_funct,Orlitz,iv_var_integr}. However, the results therein cannot be directly transferred to our setting.

To get an appropriate variational principle in our settings, we need   also  the following simple version of \cite[Lemma 2.3]{iv_var_integr}: if  $S\subseteq Y$ is a non-empty set,  $\mathfrak{f},\mathfrak{g}:Y\to \Ri$ are   functions bounded below on $S$, and $\delta >0$ is arbitrary, then
\begin{equation}\label{eq:new-lema4}
 \Omega_{\mathfrak{f}}^S(\delta)\cap \Omega_{\mathfrak{g}}^S (\delta)\neq \varnothing \Rightarrow \Omega_{\mathfrak{f}+\mathfrak{g}}^S (\delta)\subset \Omega_{\mathfrak{f}}^S(3\delta)\cap \Omega_{\mathfrak{g}}^S (3\delta).
 \end{equation}
Here for a function $\mathfrak{f} :Y\to \Ri$ we denote the $\delta\argmin$ of $\mathfrak{f}$ on $S$ by
 \[
 \Omega_{\mathfrak{f}}^S(\delta):=\{ v\in S: \mathfrak{f}(v)\le \inf \mathfrak{f}+\delta \}.
 \]

\section{Variational principle}\label{sec:VP}

\begin{thm}\label{VP}
  Let $(X,\|\cdot\|)$ be a Banach space, and let $Y=\L_2([0,1],X)$. Let   $f:X\to \Ri$ be a   lower semicontinuous and bounded below  function such that $0\in \dom f$.

  Then for any $\eps >0$ there exists $p\in P$ such that $ \|p\|_\infty < \eps$ and
  $H_p$   attains its minimum on $Y$.
\end{thm}

\begin{proof}
Let us assume without loss of generality  that $f $ is bounded below on  $X$ by zero.

Since for any $p\in P$,
\[
H_p(v)=\int_0^1 (\|v(t)\|^2 +p(v(t)) +f(\varphi(v)(t)))\, dt,
\]
for $v \equiv 0$, we get
\[
H_p(0)=\int_0^1  f(\varphi(0))\, dt= \int_0^1  f(0)\, dt =f(0).
\]
Therefore,
\[
\inf_{v\in Y} H_p(v)\le H_p(0)=f(0).
\]

Fix $r>0$ such that $r^2\ge f(0)+1$, and set $S:=rB_Y$.
Then,
\begin{equation}\label{globmin}
H_p(v)=\int_0^1 (\|v(t)\|^2 +p(v(t)) +f(\varphi(v)(t)))\, dt \ge \int_0^1 \|v(t)\|^2 \, dt =\|v\|_Y^2.
\end{equation}

Note first that $\inf_{v\in S} H_p(v) = \inf_{v\in Y} H_p(v)$: any $v\not\in S$ satisfies $\|v\|^2>r^2\ge f(0)+1\ge \inf_{v\in Y}H_p(v)+1$, hence $H_p(v)>\inf_{v\in Y}H_p(v)+1$, so $\ds\inf_{v\in Y\setminus S} H_p(v) \ge \inf_{v\in Y} H_p(v)+1$ and the infimum on $Y$ is equal to the infimum on $S$.

Moreover, for $0\le \delta <1$ and any $v\not\in S$ we have $H_p(v)>\inf_{v\in Y}H_p(v)+1>\inf_{v\in Y}H_p(v)+\delta$, so $v\not\in \delta\argmin H_p$; that is, $\delta\argmin H_p\subset S$. Combining, $\Omega^S_{H_p}(\delta)\subset \delta\argmin H_p$ for all $p\in P$: if $v_0\in S$ with $H_p(v_0)\le \inf_{v\in S} H_p(v)+\delta = \inf_{v\in Y} H_p(v)+\delta$, then $v_0\in \delta\argmin H_p$. In particular, if $H_p$ attains its minimum on $S$ at some $\bar v$, then $H_p$ attains its minimum on $Y$ at the same $\bar v$.

 \smallskip

Now, consider for $n\in\N$ the subset $D_n$ of $\cal{P}$ defined by
$$
    D_n:=\left\{ \mathfrak{p}\in {\cal{P}}:  \exists t>0  :\ \mu\left(\Omega_{H_p}^S(t)\right) < \frac{1}{n}\right\}.
$$
We will show that $D_n$ is a dense and open set in $(\cal{P},\rho)$.

Fix $n\in\N$. Let $\mathfrak{p}\in D_n$ be arbitrary and let $\beta > 0$ be such that
$$
  \mu\left(\Omega_{H_p}^S(3\beta)\right) < \frac{1}{n}.
$$
For any $\mathfrak{h}\in  \cal{P}$  such that $\rho (\mathfrak{h},0)<\beta/2r^2$, and any $v_1,v_2\in S$, by Proposition~\ref{prop:dominate-uniform} we have
$\mathfrak{h}(v_1) - \mathfrak{h}(v_2) \le 2\sup_{v\in S} \mathfrak{h}(v) \le 2r^2\rho(\mathfrak{h},0)<\beta$, therefore, $\Omega_\mathfrak{h}^S(\beta) = S$. From \eqref{eq:new-lema4} for $H_p$ and $\mathfrak{h}$ it follows that
$$
  \Omega^S_{H_{p+h}}(\beta) \subset \Omega_{H_p}^S(3\beta),
$$
thus $\mu (\Omega^S_{H_{p+h}}(\beta)) < 1/n$ and $\mathfrak{p}+\mathfrak{h}\in D_n$. So, $D_n$ is open.

\smallskip

Let now $\mathfrak{h}\in \cal{P}$ be arbitrary. Fix an arbitrary $\varepsilon \in (0,1/n)$.
Let $\delta > 0$ be provided by Proposition~\ref{prop:41}. Fix
$$
  v_0 \in \Omega_{H_h}^S (\delta),
$$
and let $ \mathfrak{p}\in \cal{P}$ satisfy the conclusion of the Proposition~\ref{prop:41} for $v_0$. Then Proposition~\ref{prop:41}$(b)$ implies that $v_0\in \Omega_{\mathfrak{p}}^S(\delta)$. From  \eqref{eq:new-lema4} it follows that
$$
  \Omega_{H_{h+p}}^S(\delta) \subset \Omega_{\mathfrak{p}}^S (3\delta).
$$
Since $\min_Y \mathfrak{p} = \mathfrak{p}(0)= 0$ and $0\in S$, we have $\Omega^S_\mathfrak{p}(3\delta) = \{v\in S:\mathfrak{p}(v)\le 3\delta\}\subset (3\delta)\argmin \mathfrak{p}$, so by monotonicity of $\mu$ and Proposition~\ref{prop:41}$(c)$, $\mu(\Omega^S_\mathfrak{p}(3\delta))\le \mu((3\delta)\argmin\mathfrak{p}) <\varepsilon$. Therefore $\mu (\Omega_{H_{h+p}}^S(\delta)) < \varepsilon$.
As $\varepsilon<1/n$, this means that $\mathfrak{h}+\mathfrak{p}\in D_n$. By  Proposition~\ref{prop:41}$(a)$, $\rho(\mathfrak{p},0)<\varepsilon$, and the distance from $\mathfrak{h}$ to $D_n$ is smaller than $\eps$. In other words, $D_n$ is dense in~$(\cal{P},\rho)$.

Sets $D_n$, $n\in \N$  are open dense sets in the complete metric space $({\cal{P}},\rho)$. By Baire Category Theorem, the set $\ds D:=\bigcap_{n\in \N} D_n$ is a dense set.

Let us fix $\mathfrak{p}\in D$. Then for any $n\in \N$ there is $t_n>0$ such that $\mu (\Omega_{H_p}^S(t_n))<1/n$. Take a sequence $\{t_n'\}_{n=1}^\infty$ with $0 < t_n' \le t_n$ that decreases monotonically to zero. Then $\mu (\Omega_{H_p}^S(t_n'))<1/n$. The sets  $A_n=\Omega_{H_p}^S(t_n')$, $n\in \N$  are  nested sets with $\mu(A_n)<1/n$, which, because of the sequential $\tau$ lower semicontinuity of $H_p$, are   sequentially $\tau$-closed.

From Corollary~\ref{cor:Cantor}   applied  to $(S,\tau)$ with the metric $d$ induced by the  norm in $Y$ for the sets $A_n$, $n\in \N$, it follows that the set $\ds A:=\bigcap_{n\in \N} A_n\neq\varnothing$. Let $\bar v\in A$. This means that $\bar v\in S$ and $\ds H_p(\bar v)\le \inf_{v\in S} H_p( v) +t_n'$, for all $n\in \N$. Hence $H_p(\bar v)= \min_{v\in S} H_p(v)=\min_{v\in Y} H_p(v)$, that is, $H_p$ attains its minimum on $Y$.

To conclude,  observe that since $D$ is dense in $({\cal{P}},\rho)$, for a given $\eps >0$  one can find a  $\mathfrak{p}\in D$   such that $\eps>\rho(\mathfrak{p},0)=\bar\rho(p,0)=\|p\|_\infty$. The proof is completed.
\end{proof}

\section{Main result}\label{sec:main}

Here is the promised extension of Theorem~\ref{thm:iz}.

\begin{thm}\label{main}
    Let $(X,\|\cdot\|)$ be a Banach space. Let $f: X \rightarrow \mathbb{R} \cup\{\infty\}$ be  lower semicontinuous, bounded below and such that $0\in \dom f$.
Consider the optimization problem $\left(P_{\|\cdot\|}\right)$ defined by \eqref{problem-1}.

For each $\varepsilon>0$ there is an equivalent norm $|\cdot|$ on $X$ such that
\begin{equation}\label{equiv}
\|\cdot\| \leq|\cdot| \leq(1+\varepsilon)\|\cdot\|
\end{equation}
and the corresponding problem $\left(P_{|\cdot|}\right)$ has a solution.
\end{thm}

\begin{proof}
Fix arbitrary $\eps >0$. From Theorem~\ref{VP} it follows that   there exists a squared  continuous seminorm $p$ on $X$ such that $\|p\|_\infty <\eps$ and
$H_p$ attains minimum on $Y$.

Setting $|\cdot|:=\sqrt{\|\cdot\|^2 +p(\cdot)}$ one gets a norm such that the problem $(P_{|\cdot|})$ has a solution.

Since for all $x\in X$, $\|x\|\le \sqrt{\|x\|^2+p(x)}=|x|$,  and for $x\neq 0$,
\[
|x|= \sqrt{\|x\|^2+p(x)}= \sqrt{ \|x\|^2 +\|x\|^2 p\left(\frac{x}{\|x\|}\right)} \le (\sqrt{1+\eps})\|x\|\le (1+\eps)\|x\|,
\]
the estimate \eqref{equiv} holds.
\end{proof}

\section{Concluding remarks}\label{sec:concl}

It might be interesting to apply our technique to the  general case of the problem \eqref{problem}, namely
\[
\min_{v\in Y} I(v),
\]
where $\ds I(v)= \int_0^1 G(v(t),u(t),t)\, dt$ with $u(t)=\varphi(v)(t)$.

    Our method requires sequential $\tau$ lower semicontinuity of $I$, so it would be good to have sufficient conditions for it in terms of $G$.

    Of course, even better if $I$ is weakly lower semicontinuous. Provided $X$ is finite dimensional, under which conditions on $G$ the value function $I$ is weakly lower semicontinuous? The classical answer (Tonelli, covered in \cite{dacorogna_direct_methods,giusti_direct_methods}): For $I(v) = \int_0^1 G(v(t), u(t), t)\, dt$ with $u = \varphi(v)$, weak lower semicontinuity typically requires $G$ to be lower semicontinuous in $(v, u)$, and  convex in $v$ (the ``velocity'' variable) with a growth/coercivity condition on it. So, our  case $G(v, u, t) = \|v\|^2 + f(u)$ is favorable: $\|v\|^2$ is convex in $v$, and $f(u)$ depends only on ``position'' variable $u$, so only lower semicontinuity of $f$ is needed. When $X$ is an infinite dimensional space  even with convexity in $v$, weak convergence $v_n \to_w v$ does not ensure that $\varphi(v_n) \to \varphi(v)$ in norm as $\varphi$ is not compact. This is exactly why $\tau$ was needed -- it forces the norm convergence of $\varphi(v_n)$ by definition.

    Therefore, in the case of infinite dimensional $X$,  some conditions to check in future research are, for example, $G$ lower semicontinuous in $(v,u)$, convex in $v$, with growth conditions. Would these allow perturbed minimization?

\bigskip    
    
\textbf{Funding.} The work is supported by the Bulgarian National Science Fund under grant No KP-06-H92/6 signed December 8, 2025.

\bigskip    
    
\textbf{Acknowledgements.}  The authors are extremely grateful to Dr. Milen Ivanov for his constant support and encouragement during the process of working on this article.

\bibliography{ref}
\bibliographystyle{plain}

\end{document}